\documentclass[11pt,reqno]{amsart}
\usepackage{amsmath}
\usepackage{amsfonts}

\numberwithin{equation}{section}
\usepackage{times}
\usepackage{amsmath,amsfonts,amstext,amssymb,amsbsy,amsopn,amsthm,eucal}
\usepackage{txfonts}
\usepackage{dsfont}
\usepackage{graphicx}   
\usepackage{hyperref}
\usepackage{accents}
\usepackage{enumerate}
\usepackage{xcolor}

\usepackage{verbatim}

\newcommand{\Ric}{{\rm Ric}}

\newcommand{\Vol}{{\rm Vol}}

\newcommand{\cH}{\mathcal{H}}

\newtheorem{theorem}{Theorem}[section]

\newtheorem{lemma}[theorem]{Lemma}

\theoremstyle{definition}
\newtheorem{definition}[theorem]{Definition}
\theoremstyle{remark}
\newtheorem{remark}{Remark}[section]
\theoremstyle{remark}

\theoremstyle{remark}

\theoremstyle{remark}
\theoremstyle{remark}

\begin{document}

\title{On the scalar curvature rigidity for mainifolds with non-positive Yamabe invariant}
\date{\today}

\author{Huaiyu Zhang}
\address{H. Zhang,~~School of mathematics and statistcs,  Nanjing University of Science and Techlonogy}
\email{zhymath@outlook.com}

\author{Jiangwei Zhang}
\address{J. Zhang,~~School of mathematical Sciences,  Zhejiang University}
\email{zhangjiangwei@zju.edu.cn}

\begin{abstract}

In this paper, we study scalar curvature rigidity of non-smooth metrics on smooth manifolds with non-positive Yamabe invariant. We prove that if  the  scalar curvature is not less than the Yamabe invariant in distributional sense, then the manifold must be  isometric to an Einstein manifold. This result extends Theorem 1.4 in Jiang, Sheng and the first author (Sci. China Math. 66 (2023) no. 6, 1141-1160), from a special case where the manifolds have zero Yamabe invariant to general cases where the manifolds have non-positive Yamabe invariant.  This result depends highly on an analysis and estimates of geometric evolution equations.

\end{abstract}
\maketitle
\tableofcontents

\section{Introduction}
Low-regularity geometry with weak curvature conditions has been appearing as an important theme in Riemannian geometry. 
Sectional curvature, Ricci curvature and scalar curvature are the most fundamental and the most important curvatures in Riemannian geometry. For  sectional curvature lower bounds, Gromov, Perelman, etc., developed the  Alexandrov spaces theory, which has great applications in the solvation of Poincar\'e conjecture, see \cite{BGP92}, \cite{MT07}, etc.
For  Ricci curvature  lower bounds, there is a profound theory developed by Cheeger, Colding, Tian, etc., see \cite{CC1}, \cite{CN1}, \cite{Ti15} \cite{JN21}, \cite{CJN21}, etc. An another theory for  Ricci curvature  lower bounds was  developed by Lott, Villani, Sturm, etc., via an optimal transport approach, see \cite{LV},  \cite{St1}.

However, for scalar curvature lower bounds, it has not been well understood. Gromov suggested people to study scalar curvature lower bounds in weak sense, see \cite[Page 1118]{Gm}. And he pointed out that one could consider weak scalar curvature in distributional sense (see \cite[Page 1118, Line 11 from below]{Gm}).  An anothoer notion of weak scalar curvature lower bounds is developed by by using Ricci flow (see \cite{Bm} and \cite{PBG}). Jiang, Sheng and the first author tied the notion developed by Bamler and Burkhardt-Guim to the notion of distributional scalar curvature in \cite{JSZ23}. In \cite{LL15}, Lee-LeFloch proved a positive mass theorem for distributional scalar curvature. In \cite{JSZ22}, Jiang, Sheng and the first author improved some of the results in \cite{LL15}.  the first author partly solved the Yamabe problem for distributional scalar curvature in \cite{Zh23}. For more work in this topic, see \cite{LM07}, \cite{LS15}, \cite{SW11}, \cite{LS21}, \cite{McSz}, \cite{Miao2002}, \cite{ST02}, \cite{LNN23}, etc.

Particularly, Schoen wanted to generalize the following theorem,  involving scalar curvature lower bounds, from smooth metrics to non-smooth metrics:
\begin{theorem}[\cite{KW75}, \cite{Sc89}]\label{thmY222}
Let $M^n$ be a compact manifold with $\sigma(M)\le0$, where $\sigma(M)$ is the Yamabe invariant of $M$, and $g$ be a  smooth metric on $M$ with unit volume  such that $R_g\ge \sigma(M)$ pointwisely on $M$. Then $g$ is Einstein with $R_g=\sigma(M)$.
\end{theorem}

The   Yamabe invariant here is an invariant defined on smooth differential manifolds, it appears to be the key concept in the resolvation of  prescribing scalar curvature problem (see \cite{KW75}). See \eqref{dnfsigma} for its formal definition. 

Schoen considered such a question: In Theorem \ref{thmY222}, if $g$ admits a singular set, and $R_g\ge\sigma(M)$ only holds away from the singular set, can we still deduce that $g$ is Einstein? In particular, he conjectured that if $g$ is uniformly Euclidean near the singular set, $R_g\ge0\ge\sigma(M)$ away from the singular set, and the singular set is a submanifold of codimensional at least $3$, then $g$ extends smoothly to the whole manifold to an Ricci flat metric (see \cite[Conjecture 1.5]{LiMa}).

This question is still open. \cite{LiMa} confirms Schoen's conjecture for $3$-manifolds in the case of isolated singularity. If the metric is $W^{1,p}(p>n)$ near the singular set and the singular set is small, \cite{ST18} proved that the metric is Einstein away from the singular set, which almost solved Schoen's conjecture in  $W^{1,p}(p>n)$ case. In \cite{JSZ23}, Jiang, Sheng and the first author improved this result and give an optimal condition of the singular set in the case $p=+\infty$. If the metric is $W^{1,n}$ or $C^0$ near the singular set and the singular set is small, Chu-Lee and  Lee-Tam  proved the metric is Einstein away from the singular set respectively in  \cite{CL22} and \cite{LT21}, which almost solved Schoen's conjecture in  $W^{1,n}$ or $C^0$ case.

Motivated by Gromov's suggestion in \cite{Gm} and Schoen's question in \cite{LiMa},  we study a more radical case in this paper. In contrast with the work above, in which there metrics are still smooth away from a small singular set, our metrics, in this paper,  could be non-smooth on the whole  manifold. Correspondingly, the scalar curvature lower bounds are assumed in distributional sense. Our main theorem is:

\begin{theorem}\label{thmY2}
Let $M^n$ be a compact manifold with $\sigma(M)\le0$, where $\sigma(M)$ is the Yamabe invariant of $M$, and $g$ be a  $W^{1,p}(n< p\le \infty)$ metric on $M$ with unit volume  such that $R_g\ge \sigma(M)$ in distributional sense. Then $(M,g)$ is isometric to an Einstein manifold with scalar curvature equaling to $\sigma(M)$.
\end{theorem}

\begin{remark}\label{rmk1.1}

Theorem \ref{thmY2} confirms Schoen's conjecture for  $W^{1,p}(n< p\le \infty)$ metrics. In fact,  if $g$ is $C^2$ away from a closed subset $\Sigma$, whose Hausdorff measure satisfis $\cH^{n-\frac{p}{p-1}}(\Sigma)<\infty$ for $n<p<\infty$ or $\cH^{n-1}(\Sigma)=0$ for $p=\infty$, and if $R_g\ge a$ pointwisely away from $\Sigma$, then we can deduce that $R_g\ge  a$ in distributional sense, see \cite[Lemma 2.7]{JSZ22}.
\end{remark}

\begin{remark}
Jiang, Sheng and the first author have proved the  special case $\sigma(M)=0$, see \cite[Theorem 1.4]{JSZ23}. However, our work in \cite{JSZ23} does not directly solve the general case  $\sigma(M)\le0$. To prove the general case $\sigma(M)\le0$, in this paper, we must do more work. Particularly, we must improve the estimate in \cite[Theorem 1.1]{JSZ23} to a better estimate. This improvement is given in our Lemma \ref{mthm2}. Moreover, we must use a modified flow in \cite{LT21}, rather than simply use the Ricci flow in \cite{JSZ23}.
\end{remark}

\begin{remark}
Since in this paper we extend the result by Jiang, Sheng and the first author in \cite{JSZ23} from $\sigma(M)=0$ to $\sigma(M)\le 0$. One might ask for similar results for $\sigma(M)>0$. In fact, the corresponding question for $\sigma(M)>0$ is nonsense, since the resolvation of the prescribing scalar curvature problem tells that for any $M$ with $\sigma(M)>0$, any smooth function is the scalar curvature of some metrics on $M$, see \cite{KW75}.
\end{remark}

\vskip 3mm
\noindent
\textbf{Organization:}

In section 2, we give some prelimilaries.
In section 3, we construct an auxiliary function and give some estimate, which is of great importance in our proof of Theorem \ref{thmY2}.
In section 4, we study scalar curvauture lower bounds along Ricci flow, where the initial metric only has scalar curvature lower bounds in distributional sense.
In section 5, based on a modified flow in \cite{LT21},  we prove Theorem \ref{thmY2}.

\section{Prelimilaries}
\subsection{Distributional scalar curvature}

Due to the lack of the second order derivative, singular metrics do not have the concept of curvature in classical sense. However, inspired by the distributional theory (or theory of  generalized functions), in which the derivative exists even for very weird functions, it is naturally to consider derivative and curvature in distributional sense for singular metrics.

Let $M^n$ be a compact smooth manifold. Fix an arbitrary smooth background metric $h$ on $M$, for any tensors field $T$ on $M$ its Sobolev norm $W^{k,q}(M)$ is defined natrually as:
\begin{align}
\|T\|_{W^{k,q}(M)}:=\sum_{s=0}^k \int_M |\tilde \nabla^k T|_h d\mu_h.
\end{align}

Here and below $\tilde\nabla$, $|\cdot|_h$ and $d\mu_h$ will denote the Levi-Civita connection, the norm and the volume form respectively taken with respect to $h$. Although the $W^{k,q}(M)$ norm  depends on the background metric $h$, the norms for different $h$ are all equivalent and the $W^{k,q}(M)$ space does not   depend on it.

Therefore, a $W^{k,q}$ metric $g$, or a metric $g\in W^{k,q}(M)$, means a symmetric and positive definite $(0,2)$ tensor field on $M$ with finite $W^{k,q}(M)$ norm.

In this paper we focus on $W^{1,p}(n<p\le+\infty)$ metrics. For any metric $g\in W^{1,p}(M)$, its distirbutional scalar curvature $R_g$ is defined as (see \cite{JSZ23,LM07,LL15,LS15}, etc.):
\begin{align}
	\langle R_g,\varphi\rangle :=\int_M \left(-V\cdot \tilde\nabla \left(\varphi \frac{d\mu_g}{d\mu_h}\right)+F\varphi \frac{d\mu_g}{d\mu_h}\right)d\mu_h, \forall \varphi\in C^\infty(M),
\end{align}
where $V$ and $F$ is a vector field and a function respectively, defined by:
\begin{align}
\Gamma_{ij}^k&:=\frac{1}{2}g^{kl}\left(\tilde\nabla_i g_{jl}+\tilde\nabla_j g_{il}-\tilde\nabla_l g_{ij}\right), \\
	V^k&:=g^{ij}\Gamma_{ij}^k-g^{ik}\Gamma_{ji}^j=g^{ij}g^{k\ell}(\tilde\nabla_jg_{i\ell}-\tilde\nabla_{\ell}g_{ij}),\\
	F&:= {\rm{tr}}_g\Ric_h -\tilde \nabla_kg^{ij}\Gamma_{ij}^k+\tilde\nabla_kg^{ik}\Gamma_{ji}^i+g^{ij}\left(\Gamma_{k\ell}^k\Gamma_{ij}^\ell-\Gamma_{j\ell}^k\Gamma_{ik}^\ell\right),
\end{align}
where $\Ric_h $ is the Ricci curvature tensor of $h$.

 This definition is independent of the background metric $h$, see \cite{LL15}. Moreover, if $g$ is smooth, this definition of $\langle R_g,\varphi\rangle$ is just the integral $\int_M R_g\varphi d\mu_g$. See also  \cite{JSZ22,Zh23}, etc., for more properties of the distirbutional scalar curvature.

\subsection{Mollification of the metric and estimate on the scalar curvature}

Any $W^{1,p}$ metric could be mollified to a smooth metrics family, which converges to it in $W^{1,p}$ topology. Concretly, we have the following lemma:
\begin{lemma}[Lemma 4.1 in \cite{grant2014positive}]\label{lm2.1}
Let $M^n$ be a compact smooth manifold with  a $W^{1,p}(n< p\le\infty)$ metric $g$ on it, then there exists a family of smooth metrics $g_\delta,\delta>0$, such that $g_\delta$ converges to $g$ in $W^{1,p}$ topology as $\delta\to 0^+$.
\end{lemma}
\begin{remark}
In \cite[Lemma 4.1]{grant2014positive}, the lemma is only claimed for $W^{2,\frac{n}{2}}$ metrics. However, their proof indeed works for general $W^{k,q}$ case, especially for our $W^{1,p}$ case.
\end{remark}

The following lemma established by Jiang, Sheng and the first author shows that the distributional scalar curvature  functional is continuously dependent of the metric in a certain sense:
\begin{lemma}[Lemma 2.2 in \cite{JSZ23}]\label{lm2.2}
	Let $M^n$ be a compact smooth manifold.   Suppose $g_\delta$ is a family of metrics which converges to $g$ in $W^{1,p}$ topology as $\delta\to 0^+$, then we have that for any $\epsilon>0$, there exists $\delta_0=\delta_0(g)>0$, such that 
\[|\langle R_{g_ \delta},u\rangle-\langle R_g,u\rangle|\leq \epsilon\|u\|_{W^{1,\frac{n}{n-1}}(M)},\forall u\in C^\infty(M),\forall \delta\in (0,\delta_0).\]
where $R_{g_\delta}$ is the scalar curvature of $g_\delta$.
\end{lemma}

\subsection{Estimates on Ricci flow}

The Ricci flow, introduced by Hamilton in \cite{Hamilton1982},  is defined as follows:
\begin{definition}[Ricci flow]The Ricci flow on $M$ is a family of metrics $g(t)$ such that \[\frac{\partial}{\partial t}g(t)=-2\Ric_{g(t)},\] where ${\Ric_{g(t)}}$ is the Ricci curvature tensor of $g(t)$.  
\end{definition}

Though Hamilton only introduced Ricci flow with smooth initial metrics, it is quite useful to consider Ricci flow with singular initial metrics. For our use, we mainly need the  following theorem, given by Jiang, Sheng and the first author, which considers Ricci flow with $W^{1,p}(n<p\le+\infty)$ initial metrics:
\begin{lemma}[Theorem 3.2 in \cite{JSZ23}]\label{thm66.2} \label{mainsec6}
There exists an $\epsilon(n)>0$ such that, for any compact $n$-manifold $M$ and any $ W^{1,p}(n<p\le +\infty)$ metric $\hat g$ on $M$, there exists a $T_0=T_0(n,\hat g)>0$ and a Ricci flow $g(t)\in C^{\infty}(M\times(0,T_0])$, such that
\begin{itemize}
\item[(1)]$\lim_{t\to 0}d_{GH}((M,g(t)),(M,\hat g))=0.$
\item[(2)] $|{\rm{Rm}}(g(t))|(t)\leq \frac{C(n,\hat g,p)}{t^{\frac{n}{4p}+\frac{3}{4}}}$, $\forall t\in(0,T_0]$.
\item[(3)] $\int_0^{T_0}\int_M |{\rm{Rm}}( g(t))|^2d\mu_{g(t)}dt\le C(n,\hat g,p)$,
\end{itemize}
where $d_{GH}$ is the Gromov-Hausdorff distance, and $C(n,\hat g,p)$ is a positive constant independent of $t$.
\end{lemma}

We also need to consider the $h$-flow, which is equivalent to the Ricci flow after a family of diffeomorphisms. It was firstly introduced by Simon in \cite{Si02}, in order to study the Ricci flow with $C^0$ initial two metrics.

Before we give the definition of $h$-flow, we firstly give the definition of $(1+\delta)$-fairness between metrics.
\begin{definition}
Given a constant $\delta\geq 0$, a metric $h$ is called to be $(1+\delta)$-fair to $g$, if $h$ is $C^\infty$, 
\[\sup_M|\tilde\nabla^j {\rm{Rm}}(h)|=L_j<\infty,\]
and \[(1+\delta)^{-1}h\leq g\leq (1+\delta) h \quad on\ M.\]
\end{definition}
Here and below, $\tilde \nabla$ means the covariant derivative taken with respect to $h$.

\begin{definition}\label{dfn8}[$h$-flow] For any background smooth metrics $h$, the $h$-flow is a family of metrics $g(t)$ satisfies \[\frac{\partial}{\partial t}g_{ij}(t)=-2R_{g(t);ij}+ \nabla_iV_j+ \nabla_jV_i,\]
where the derivatives are taken with respect to $g(t)$, $R_{g(t);ij}$ is the Ricci curvature of $g(t)$, \[V_j=g_{jk}(t)g^{pq}(t)(\Gamma_{pq}^k(t)-\tilde\Gamma_{pq}^k),\]
and $\Gamma(t)$ and $\tilde\Gamma$ are the Christoffel symbols of $g(t)$ and $h$ respectively. 
\end{definition}

Simon proved such an existence results:

\begin{lemma}[Theorem 1.1 in \cite{Si02}]\label{thm2.2}
There exists an $\epsilon(n)>0$ such that, for any compact $n$-manifold $M$ with a complete $C^0 $ metric $\hat g$ and a $C^\infty$ metric $h$ which is $(1+\frac{\epsilon(n)}{2})$-fair to $\hat g$, there exists a $T_0=T_0(n,k_0)>0$ and a family of metrics $g(t)\in C^{\infty}(M\times(0,T_0]),t\in (0,T_0]$ which solves $h$-flow for $t\in(0,T_0]$, $h$ is $(1+\epsilon(n))$-fair to $g(t)$, and
\begin{itemize}
\item[(1)] \begin{align*}
\lim_{t\to0^+}\sup_{x\in M}|g(x,t)-\hat g(x)|=0,
\end{align*}
\item[(2)] \begin{align*}
\sup_{M}|\tilde\nabla^ig(t)|\leq\frac{c_i(n,h)}{t^{i/2}}, \forall t\in(0,T_0],i\ge 1,\end{align*}
\end{itemize}
where the derivatives and the norms are taken with respect to $h$.

\end{lemma}

To apply the flow to prove our results, we will let $h$ be $(1+\frac{\epsilon(n)}{2})$-fair to $\hat g$. The existence of such $h$ is ensured  by the following lemma:

\begin{lemma}[\cite{Si02}]\label{rmkh}
Let $M$ be a compact manifold, for any   $C^0$ metric $g$ on $M$, and any $0<\delta<1$, there exists a $C^\infty$ metric $h$ which is $(1+\delta)$-fair to $g$.
\end{lemma}

In the case of $W^{1,p}(n<p\le+\infty)$ initial metric, Lemma \ref{thm2.2} could be improved:
\begin{lemma}[Theorem 3.11 in \cite{JSZ23}]\label{thm6.2}
In the condition of Lemma \ref{thm2.2},  if $\hat g$ is  $W^{1,p}(n<p\le+\infty)$ on $M$, then there exists a $T_0=T_0(n,h,\|\hat g\|_{W^{1,p}(M)},p)$, such that $g(t)$, $ t\in(0,T_0]$ is the $h$-flow with initial metric $\hat g$, and
\begin{itemize}
\item[(1)]$\int_M|\tilde\nabla g(t)|^pd\mu_h\leq10\int_M|\tilde\nabla \hat g|^pd\mu_h$, $\forall t\in(0,T_0]$,
\item[(2)]$|\tilde\nabla g|(t)\leq \frac{C(n,h,\|\hat g\|_{W^{1,p}(M)},p)}{t^{\frac{n}{2p}}}$, $\forall t\in(0,T_0]$,
\item[(3)]$|\tilde\nabla^2 g|(t)\leq \frac{C(n,h,\|\hat g\|_{W^{1,p}(M)},p)}{t^{\frac{n}{4p}+\frac{3}{4}}}$, $\forall t\in(0,T_0],$
\end{itemize}
where  $C(n,h,\|\hat g\|_{W^{1,p}(M)},p)$ is a positive constant depends only on $n,h,\|\hat g\|_{W^{1,p}(M)},p$, and does not depend on $t$.
\end{lemma}

The Ricci flow in Lemma \ref{mainsec6} just comes from the $h$-flow in Lemma \ref{thm6.2}. In fact, they are equivalent after a certain family of diffeomorphisms.

\section{An auxiliary function along Ricci flow}
In this section, we consider a compact smooth Riemannian manifold $(M^n,\hat g)$ and the Ricci flow $g(t)(t\in[0,T_0])$ with initial metric $\hat g$. 

For any positive time $T\in(0,T_0]$ and any  $\tilde \varphi\in C^\infty(M)$, we want to construct an auxiliary function $\varphi(x,t)$ defined on $M\times [0,T]$, such that $\varphi(T,\cdot)=\tilde \varphi(\cdot)$, and the integral
\begin{align}\int_M \left(R_{g(t)} -a\left(1-\frac{2a}{n}t\right)^{-1}\right) \varphi(\cdot,t)d\mu_{g(t)}
\end{align} is monotonously increasing with respect to $t$.

Moreover, we also need some estimate for $\varphi(x,t)$. In fact, we have the following lemma, which is the main result of this section:
 \begin{lemma}\label{prop_varphit}
	Let $(M^n,\hat g)$ be a compact smooth Riemannian manifold and $g(t)(t\in[0,T_0])$ be the Ricci flow with initial metric $\hat g$. For any non-positive constant $a$, any $T\in(0,T_0]$ and any  $\tilde \varphi\in C^\infty(M)$, there exists a function $\varphi(x,t)\in C^\infty(M\times [0,T])$, such that
\begin{enumerate}
	\item $\varphi(\cdot,T)=\tilde \varphi(\cdot),\  {\text{on}}\ M$.
	\item  $\int_M \left(R_{g(t)} -a\left(1-\frac{2a}{n}t\right)^{-1}\right) \varphi(\cdot,t)d\mu_{g(t)}$ is monotonously increasing with respect to $t$.
	\item $\varphi(\cdot,t)\le C(n,h,p,\|\hat g\|_{W^{1,p}(M)},\|\tilde\varphi\|_{L^\infty},a)$, $\forall t\in[0,T].$
	\item  $\|\varphi(\cdot,t)\|_{W^{1,\frac{n}{n-1}}(M)}\le C(n,h,p,\|\hat g\|_{W^{1,p}(M)},\|\tilde\varphi\|_{L^\infty},a)$,  $\forall t\in[0,T].$
\end{enumerate}
\end{lemma}

\begin{remark}
Though Lemma \ref{prop_varphit} only claims results for $a$ non-positve, our proof essentially give similar results for $a$ positve. The matter is that the term $\left(1-\frac{2a}{n}t\right)^{-1}$ in item (2) would divergence to infinity at $t=\frac{2}{2a}$. This divergence phenomenen is expectable by noting that the Ricci flow with the round sphere with scalar curvature $a>0$ as its initial metric collapses to a single point at time $t=\frac{2}{2a}$. However, the case of $a$ non-positve is good enough for our use.
\end{remark}

Now we are going to construct the function $\varphi$ in the lemma above.
Since $(M^n,\hat g)$ is a compact smooth Riemannian manifold, it is known that the heat kernel along its Ricci flow $g(t)(t\in[0,T_0]$ exists. The heat kernel is a funciton $K(y,s;x,t)$, $y,x\in M$, $0\le t<s\le T_0$, such that
\begin{align}
(\partial_s-\Delta_{g(s);y})K(y,s;x,t)=0,\forall\  {\text{fixed}}\  x\in M, t\in (0,T_0],
\end{align}
and 
\begin{align}
\lim_{s\to t^+}K(y,s;x,t)=\delta_x(y), \forall\  {\text{fixed}}\  x\in M,
\end{align}
where $\Delta_{g(s);y}$ denotes the Laplacian taken to the variable $y$ and with respect to the metric $g(s)$, $\delta_x(\cdot)$ denotes the Dirac Delta functional at $x$.

Moreover, 
\begin{align}\label{eqK}
(\partial_t+\Delta_{g(t);x}+R_{g(t)}(x))K(y,s;x,t)=0,\forall\  {\text{fixed}}\  y\in M, s\in (0,T_0],
\end{align}
and 
\begin{align}
\lim_{t\to s^-}K(y,s;x,t)=\delta_y(x), \forall\  {\text{fixed}}\  y\in M,
\end{align}
where $R_{g(t)}$ is the scalar curvature of $g(t)$.

Now, for any  $a\in \mathbb{R}, T\in (0,T_0]$ and $\tilde \varphi\in C^\infty(M)$, we can construct the function $\varphi$ f as follows:
\begin{align}\label{dfnvarphi}
&\varphi: M\times [0,T]\to M,\notag\\
(x,t)\mapsto \left(1-\frac{2a}{n}T\right)^{-2}&\left(1-\frac{2a}{n}t\right)^2\int_M \tilde\varphi(y) K(y,T;x,t)d\mu_{g(T)} (y),
\end{align}

For convenience, we denote:
\begin{align}\label{dfnat}
a(t):=a\left(1-\frac{2a}{n}t\right)^{-1}.
\end{align}

One of the advantages of $\varphi$ defined above is that it satisfies a equation which is useful for our use:
\begin{lemma}
Let $(M,g(t))$ be a Ricci flow, $K(y,s;x,t)$ be the heat kernel along $g(t)$, then
the function $\varphi$ defined in \eqref{dfnvarphi} satisfies
\begin{align}\label{eqphi0}
\partial_t\varphi=-\Delta_{g(t)}\varphi+\left(R_{g(t)}-\frac{4}{n}a(t)\right)\varphi,
\end{align}
where $a(t)$ is the real function defined in \eqref{dfnat}.
\end{lemma}
\begin{proof}
Since 
\begin{align}
\varphi(x,t)=\left(1-\frac{2a}{n}T\right)^{-2}\left(1-\frac{2a}{n}t\right)^2\int_M \tilde\varphi(y) K(y,T;x,t)d\mu_{g(T)} (y),
\end{align}

we have
\begin{align}\label{partialtphi}
&\partial_t\varphi(x,t)\notag\\
=&\left(1-\frac{2a}{n}T\right)^{-2}\partial_t \left(1-\frac{2a}{n}t\right)^2\int_M \tilde\varphi(y) K(y,T;x,t)d\mu_{g(T)} (y)\notag\\
&+\left(1-\frac{2a}{n}T\right)^{-2}\left(1-\frac{2a}{n}t\right)^2\partial_t \int_M \tilde\varphi(y) K(y,T;x,t)d\mu_{g(T)} (y)\notag\\
=&-\left(1-\frac{2a}{n}T\right)^{-2}\frac{4a}{n} \left(1-\frac{2a}{n}t\right)\int_M \tilde\varphi(y) K(y,T;x,t)d\mu_{g(T)} (y)\notag\\
&+\left(1-\frac{2a}{n}T\right)^{-2}\left(1-\frac{2a}{n}t\right)^2\partial_t \int_M \tilde\varphi(y) K(y,T;x,t)d\mu_{g(T)} (y)
\notag\\
=&-\frac{4}{n} a(t)\varphi(x,t)\notag\\
&+\left(1-\frac{2a}{n}T\right)^{-2}\left(1-\frac{2a}{n}t\right)^2\partial_t \int_M \tilde\varphi(y) K(y,T;x,t)d\mu_{g(T)} (y).
\end{align}

By \eqref{eqK}, we have
\begin{align}
\partial_t K(y,T;x,t)=(-\Delta_{g(t);x}+R_{g(t)}(x))K(y,T;x,t).
\end{align}

Thus \eqref{partialtphi} becomes
\begin{align}\label{partialtphi2}
\partial_t\varphi(x,t)
=-\frac{4}{n} a(t)\varphi(x,t)+\left(1-\frac{2a}{n}T\right)^{-2}\left(1-\frac{2a}{n}t\right)^2\int_M \tilde\varphi(y) (-\Delta_{g(t);x}+R_{g(t)}(x))K(y,T;x,t)d\mu_{g(T)} (y).
\end{align}

On the other hand, we have
\begin{align}\label{partialtphi3}
\Delta_{g(t)}\varphi(x,t)
&=\left(1-\frac{2a}{n}T\right)^{-2}\left(1-\frac{2a}{n}t\right)^2\int_M \tilde\varphi(y) \Delta_{g(t)} K(y,T;x,t)d\mu_{g(T)} (y).
\end{align}

Combining \eqref{partialtphi2} and \eqref{partialtphi3}, we have
\begin{align}\label{partialtphi4}
(\partial_t+\Delta_{g(t)})\varphi(x,t)
&=-\frac{4}{n} a(t)\varphi(x,t)+\left(1-\frac{2a}{n}T\right)^{-2}\left(1-\frac{2a}{n}t\right)^2\int_M \tilde\varphi(y) R_{g(t)}(x)K(y,T;x,t)d\mu_{g(T)} (y)\notag\\
&=-\frac{4}{n} a(t)\varphi(x,t)+R_{g(t)}(x)\left(1-\frac{2a}{n}T\right)^{-2}\left(1-\frac{2a}{n}t\right)^2\int_M \tilde\varphi(y) (x)K(y,T;x,t)d\mu_{g(T)} (y)\notag\\
&=-\frac{4}{n} a(t)\varphi(x,t)+R_{g(t)}(x)\varphi(x,t),
\end{align}
which is just the equation \eqref{eqphi0}, thus the lemma is proved.
\end{proof}

Now we can prove Lemma \ref{prop_varphit}:
\begin{proof}[proof of Lemma \ref{prop_varphit}]
Let $K(y,s;x,t)$ be the heat kernel along $g(t)$, and $\varphi$ be the function defined in \eqref{dfnvarphi}.

\textbf{For (1)}, it is known that the convolution $\int_M \tilde\varphi(y) (x)K(y,T;x,t)d\mu_{g(T)} (y)$ is smooth and
\begin{align}
\int_M \tilde\varphi(y) (x)K(y,T;x,t)d\mu_{g(T)} (y)\Big|_{t=T}=\tilde\varphi(x).
\end{align}

On the other hand, we have
\begin{align}
\left(1-\frac{2a}{n}T\right)^{-2}\left(1-\frac{2a}{n}t\right)^2|_{t=T}=1.
\end{align}

Thus by \eqref{dfnvarphi} we have
\begin{align}
\varphi(x,T)=1\tilde\varphi(x)=\tilde\varphi(x),
\end{align}
which proves (1).

\textbf{For (2)}, we calculate that
\begin{align}\label{mnt1}
&\partial_t \int_M \left(R_{g(t)} -a\left(1-\frac{2a}{n}t\right)^{-1}\right) \varphi(x,t)d\mu_{g(t)}\notag\\
=&\partial_t \int_M \left(R_{g(t)} -a(t)\right) \varphi(x,t)d\mu_{g(t)}\notag\\
=& \int_M \left(\left(\partial_tR_{g(t)} -a'(t)\right) \varphi(x,t)+\left(R_{g(t)} -a'(t)\right) \partial_t\varphi(x,t)\right)d\mu_{g(t)}+\int_M \left(R_{g(t)} -a(t)\right) \varphi(x,t)\partial_td\mu_{g(t)}.
\end{align}

And we have
\begin{align}
a'(t)=\frac{2a^2}{n}\left(1-\frac{2a}{n}t\right)^{-2}=\frac{2}{n}a^2(t).
\end{align}

Recall $g(t)$ is a Ricci flow, we have the standard evolution equation
\begin{align}
\partial_t R_{g(t)}=\Delta_{g(t)}R_{g(t)}+2|\Ric_{g(t)}|_{g(t)}^2,
\end{align}
and 
\begin{align}\label{mnt4}
\partial_t d\mu_{g(t)}=-R_{g(t)}d\mu_{g(t)},
\end{align}

Combining \eqref{mnt1}-\eqref{mnt4} and the evolution equation of $\varphi$, \eqref{eqphi0}, we have
\begin{align}\label{mnt5}
&\partial_t \int_M \left(R_{g(t)} -a\left(1-\frac{2a}{n}t\right)^{-1}\right) \varphi(x,t)d\mu_{g(t)}\notag\\
=& \int_M \left(\left(\Delta_{g(t)}R_{g(t)}+2|\Ric_{g(t)}|_{g(t)}^2 -\frac{2}{n}a^2(t)\right) \varphi(x,t)+\left(R_{g(t)} -a(t)\right)\left( -\Delta_{g(t)}\varphi+\left(R_{g(t)}-\frac{4}{n}a(t)\right)\varphi(x,t)\right)\right)d\mu_{g(t)}\notag\\
&+\int_M \left(R_{g(t)} -a(t)\right) \varphi(x,t)\left(-R_{g(t)}\right) d\mu_{g(t)}\notag\\
=& \int_M \left(\left(\Delta_{g(t)}R_{g(t)}\varphi(x,t)-R_{g(t)}\Delta_{g(t)}\varphi(x,t)\right)+\left(2|\Ric_{g(t)}|_{g(t)}^2 -\frac{2}{n}a^2(t)-\frac{4}{n}a(t)R_{g(t)}+\frac{4}{n}a^2(t)\right) \varphi(x,t)\right)d\mu_{g(t)}\notag\\
&+\int_M \left(R_{g(t)} -a(t)\right) \varphi(x,t)\left(-R_{g(t)}\right) d\mu_{g(t)}+\int_M \left(R_{g(t)} -a(t)\right) \varphi(x,t)R_{g(t)}d\mu_{g(t)}\notag\\
&+\int_M a(t)\Delta_{g(t)} \varphi(x,t)Rd\mu_{g(t)}\notag\\
=& \int_M \left(\left(\Delta_{g(t)}R_{g(t)}\varphi(x,t)-R_{g(t)}\Delta_{g(t)}\varphi(x,t)\right)+\left(2|\Ric_{g(t)}|_{g(t)}^2 -\frac{4}{n}a(t)R_{g(t)}+\frac{2}{n}a^2(t)\right) \varphi(x,t)\right)d\mu_{g(t)}\notag\\
&+\int_M a(t)\Delta_{g(t)} \varphi(x,t)Rd\mu_{g(t)}.
\end{align}

By integral by part, we have
\begin{align}
\int_M \left(\Delta_{g(t)}R_{g(t)}\varphi(x,t)-R_{g(t)}\Delta_{g(t)}\varphi(x,t)\right)d\mu_{g(t)}=0,
\end{align}
and 
\begin{align}\label{mnt7}
\int_M a(t)\Delta_{g(t)} \varphi(x,t)Rd\mu_{g(t)}=a(t)\int_M \Delta_{g(t)} \varphi(x,t)Rd\mu_{g(t)}=0.
\end{align}

Combining \eqref{mnt5}-\eqref{mnt7}, we have
\begin{align}\label{mnt8}
&\partial_t \int_M \left(R_{g(t)} -a\left(1-\frac{2a}{n}t\right)^{-1}\right) \varphi(x,t)d\mu_{g(t)}\notag\\
=& \int_M \left(\left(2|\Ric_{g(t)}|_{g(t)}^2 -\frac{4}{n}a(t)R_{g(t)}+\frac{2}{n}a^2(t)\right) \varphi(x,t)\right)d\mu_{g(t)}.
\end{align}

By Cauchy inequality and a direct calculus using a special coordinate, one has
\begin{align}
2|\Ric_{g(t)}|_{g(t)}^2\ge \frac{2}{n}R_{g(t)}^2.
\end{align}

By mean value inequality, we have
\begin{align}\label{mnt11}
-\frac{4}{n}a(t)R_{g(t)}\ge -\frac{2}{n}R_{g(t)}^2-\frac{2}{n}a^2(t)
\end{align}

Combining \eqref{mnt8}-\eqref{mnt11}, we have
\begin{align}\label{mnt12}
&\partial_t \int_M \left(R_{g(t)} -a\left(1-\frac{2a}{n}t\right)^{-1}\right) \varphi(x,t)d\mu_{g(t)}\notag\\
\ge & \int_M \left(\left(\frac{2}{n}R_{g(t)}^2-\frac{2}{n}R_{g(t)}^2 -\frac{2}{n}a^2(t)+\frac{2}{n}a^2(t)\right) \varphi(x,t)\right)d\mu_{g(t)}\notag\\
= & 0.
\end{align}

Thus $\int_M \left(R_{g(t)} -a\left(1-\frac{2a}{n}t\right)^{-1}\right) \varphi(\cdot,t)d\mu_{g(t)}$ is monotonously increasing with respect to $t$ and (2) is proved.

\textbf{For (3)},   by \eqref{dfnvarphi} we have
\begin{align}\label{phi64}
\varphi(t,x)\leq  \left(1-\frac{2a}{n}T\right)^{-2}\left(1-\frac{2a}{n}t\right)^2\|\tilde\varphi\|_{L^\infty}\int_M K(y,T;x,t)d\mu_{g(T)} (y).
\end{align}
We denote $I(t,T)=\int_M  K(y,T;x,t)d\mu_{g(T)} (y)$, then we have $\lim_{T\to t^+}I(t,T)=1$, and by the standard evolution equation $\partial_T d\mu_{g(T)} =-R_{g(T)}d\mu_{g(T)}$, we have
\begin{align}\label{p1}
\partial_TI(t,T)=\int_M\left(\Delta_{g(T);y} K(y,T;x,t)-R_{g(T)} K(y,T;x,t)\right)d\mu_{g(T)} (y).
\end{align}

By divergence theorem, we have 
\begin{align}\label{p2}
\int_M \Delta_y K(y,T;x,t) d\mu_{g(T)} (y)=0.
\end{align}

Combining Lemma \ref{mainsec6} (2), \eqref{p1} and \eqref{p2}, we have 
\begin{align}\label{p3}
\partial_T I(t,T)\le \frac{C(n,h,p,\|\hat g\|_{W^{1,p}(M)})}{T^{\frac{n}{4p}+\frac{3}{4}}}I(t,T).
\end{align}

Since $\lim_{T\to t^+}I(t,T)=1$ and ${\frac{n}{4p}+\frac{3}{4}}\in (0,1)$, by taking integral we have
\begin{align}\label{phi65}
I(t,T)\le C(n,h,p,\|\hat g\|_{W^{1,p}(M)}),\forall 0\le t<T\le T_0.
\end{align}
By (\ref{phi64}) and (\ref{phi65}), we have
\begin{align}\label{philinf1}
\varphi(t,x)\le C(n,h,p,\|\hat g\|_{W^{1,p}(M)},\|\tilde\varphi\|_{L^\infty})\left(1-\frac{2a}{n}T\right)^{-2}\left(1-\frac{2a}{n}t\right)^2,\forall t\in[0,T].
\end{align}

Let us estimate $\left(1-\frac{2a}{n}T\right)^{-2}\left(1-\frac{2a}{n}t\right)^2$. 

If $a>0$, then we have assumed $t\le T\le T_0<\frac{n}{2a}$ without loss of generality. In this case $\left(1-\frac{2a}{n}t\right)^2$ is monotonously increasing with respect to $t$, thus 
\begin{align}
\left(1-\frac{2a}{n}T\right)^{-2}\left(1-\frac{2a}{n}t\right)^2\le 1,\ {\text{if}}\ a>0
\end{align}

If $a<0$, then we have
\begin{align}
\left(1-\frac{2a}{n}T\right)^{-2}\left(1-\frac{2a}{n}t\right)^2\le \left(1\right)^{-2} \left(1-\frac{2a}{n}T_0\right)^2\le C(a,T_0)=C(n,h,p,\|\hat g\|_{W^{1,p}(M)},a),\  {\text{if}}\ a<0
\end{align}

If $a=0$, then we have
\begin{align}
\left(1-\frac{2a}{n}T\right)^{-2}\left(1-\frac{2a}{n}t\right)^2\equiv1,\  {\text{if}}\ a=0
\end{align}

In all cases, we always have
\begin{align}\label{philinf2}
\left(1-\frac{2a}{n}T\right)^{-2}\left(1-\frac{2a}{n}t\right)^2\le C(n,h,p,\|\hat g\|_{W^{1,p}(M)},a).
\end{align}

Combining \eqref{philinf1} and \eqref{philinf2}, we have
\begin{align}
\varphi(x,t)\le C(n,h,p,\|\hat g\|_{W^{1,p}(M)},\|\tilde\varphi\|_{L^\infty},a), \forall (x,t)\in M\times[0,T],
\end{align}
which proves (3).

\textbf{For (4)}, we consider a simpler function 
\begin{align}
\psi(x,t):=\int_M \tilde\varphi(y) K(y,T;x,t)d\mu_{g(T)} (y)=\left(1-\frac{2a}{n}T\right)^2\left(1-\frac{2a}{n}t\right)^{-2}\varphi(x,t),
\end{align}
firstly.

Now we want to estimate $\int_M | \nabla_{g(t)} \psi(\cdot,t)|^2_{g(t)}d\mu_{g(t)}$, which we denote as $E(t)$. We calculate that
\begin{align}\label{a6.6}
\partial_t E(t)
=&\partial_t \int_M \psi(\cdot,t)_i\psi(\cdot,t)_j g(t)^{ij} d\mu_{g(t)}\notag\\
=&2\int_M \partial_t \psi(\cdot,t)_i\psi(\cdot,t)_j g(t)^{ij} d\mu_{g(t)}+\int_M \psi(\cdot,t)_i\psi(\cdot,t)_j \partial_t g(t)^{ij} d\mu_{g(t)}+\int_M | \nabla_{g(t)} \psi(\cdot,t)|_{g(t)}^2 \partial_t d\mu_{g(t)} \notag\\
=&\int_M \left(2 \langle   \nabla_{g(t)} \partial_t \psi(\cdot,t),  \nabla_{g(t)} \psi(\cdot,t)\rangle_{g(t)}\right)d\mu_{g(t)}+\int_M \psi(\cdot,t)_i\psi(\cdot,t)_j \partial_t g(t)^{ij} d\mu_{g(t)}+\int_M (-R_{g(t)}) | \nabla_{g(t)} \psi(\cdot,t)|_{g(t)}^2d\mu_{g(t)}.
\end{align}

From the Ricci flow equation
\begin{align}
\partial_t g(t)_{ij}=-2\Ric_{g(t);ij},
\end{align}
we have
\begin{align}
\partial_t g(t)^{ij}=2\Ric_{g(t)}^{ij}.
\end{align}

Thus \eqref{a6.6} becomes
\begin{align}\label{a6.71}
&\partial_t E(t)\notag\\
=&\int_M \left(2 \langle   \nabla_{g(t)} \partial_t \psi(\cdot,t),  \nabla_{g(t)} \psi(\cdot,t)\rangle_{g(t)}+2\Ric_{g(t)}( \nabla_{g(t)} \psi(\cdot,t),  \nabla_{g(t)} \psi(\cdot,t))-R_{g(t)} | \nabla_{g(t)} \psi(\cdot,t)|_{g(t)}^2\right)d\mu_{g(t)}.
\end{align}

By \eqref{dfnvarphi}, we have
\begin{align}\label{a671}
&\int_M \langle   \nabla_{g(t)} \partial_t \psi(\cdot,t),  \nabla_{g(t)} \psi(\cdot,t)\rangle_{g(t)} d\mu_{g(t)}\notag\\
=&\int_M -\langle   \nabla_{g(t)} (\Delta_{g(t)} \psi(\cdot,t)-R_{g(t)}  \psi(\cdot,t)),  \nabla_{g(t)} \psi(\cdot,t)\rangle_{g(t)} d\mu_{g(t)} \notag\\ 
=&\int_M \left(-\langle \nabla_{g(t)} \Delta_{g(t)} \psi(\cdot,t) ,\nabla_{g(t)} \psi(\cdot,t)\rangle_{g(t)} +\langle\nabla_{g(t)}(R_{g(t)} \psi(\cdot,t)),\nabla_{g(t)}\psi(\cdot,t)\rangle_{g(t)}\right)d\mu_{g(t)}.
\end{align}

Using the Bochner formula, \eqref{a671} becomes
\begin{align}\label{a672}
&\int_M \langle   \nabla_{g(t)} \partial_t \psi(\cdot,t),  \nabla_{g(t)} \psi(\cdot,t)\rangle_{g(t)} d\mu_{g(t)}\notag\\
=&\int_M \Big(-\frac{1}{2}\Delta_{g(t)}|\nabla_{g(t)}\psi(\cdot,t)|_{g(t)}^2+|\nabla_{g(t)}^2\psi(\cdot,t)|_{g(t)}^2+\Ric_{g(t)}(\nabla_{g(t)}\psi(\cdot,t),\nabla_{g(t)}\psi(\cdot,t))\notag \\ 
&\qquad\qquad+\langle\nabla_{g(t)}(R_{g(t)} \psi(\cdot,t)),\nabla_{g(t)}\psi(\cdot,t)\rangle_{g(t)}\Big)d\mu_{g(t)}.
\end{align}

By integral by part, we have
\begin{align}\label{a6721}
\int_M \frac{1}{2}\Delta_{g(t)}|\nabla_{g(t)}\psi(\cdot,t)|_{g(t)}^2d\mu_{g(t)}=0.
\end{align}

Combining  \eqref{a672} and \eqref{a6721} we have
\begin{align}\label{a673}
&\int_M \langle   \nabla_{g(t)} \partial_t \psi(\cdot,t),  \nabla_{g(t)} \psi(\cdot,t)\rangle_{g(t)} d\mu_{g(t)}\notag\\
=&\int_M \left(\left(|\nabla_{g(t)}^2\psi(\cdot,t)|_{g(t)}^2+\Ric_{g(t)}(\nabla_{g(t)}\psi(\cdot,t),\nabla_{g(t)}\psi(\cdot,t))\right)-R_{g(t)} \psi(\cdot,t)\Delta_{g(t)}\psi(\cdot,t)\right)d\mu_{g(t)}.
\end{align}

Since $|\Delta_{g(t)}\psi(\cdot,t)|_{g(t)}^2\le C_1(n)|\nabla_{g(t)}^2\psi(\cdot,t)|^2$, using Cauchy inequality, we have
\begin{align}\label{a674}
\int_M \left(-R_{g(t)} \psi(\cdot,t)\Delta_{g(t)}\psi(\cdot,t)\right)d\mu_{g(t)}
&\ge  -\frac{1}{2C_1(n)}\int_M |\Delta_{g(t)}\psi(\cdot,t)|_{g(t)}^2 d\mu_{g(t)}
-\frac{C_1(n)}{2}\int_M R_{g(t)} ^2\psi(\cdot,t)^2 d\mu_{g(t)}\notag\\
&\ge  -\frac{1}{2}\int_M |\nabla_{g(t)}^2\psi(\cdot,t)|^ d\mu_{g(t)}
-\frac{C_1(n)}{2}\int_M R_{g(t)} ^2\psi(\cdot,t)^2 d\mu_{g(t)}
\end{align}

Combining \eqref{a673} and \eqref{a674} we have
\begin{align}\label{a6.8}
 \int_M \langle   \nabla_{g(t)} \partial_t \psi(\cdot,t),  \nabla_{g(t)} \psi(\cdot,t)\rangle_{g(t)} d\mu_{g(t)}\ge \int_M \left(\Ric_{g(t)}(\nabla_{g(t)}\psi(\cdot,t),\nabla_{g(t)}\psi(\cdot,t))-C(n)R_{g(t)} ^2\psi(\cdot,t)^2\right)d\mu_{g(t)}.
\end{align}

Combining \eqref{a6.71} and \eqref{a6.8}, and by Cauchy inequality (see also \cite{LLZ23}, \cite{SHC23}), we have 
\begin{align}\label{aab1}
\partial_tE(t)&\ge \int_M \left(4\Ric_{g(t)}(\nabla_{g(t)}\psi(\cdot,t),\nabla_{g(t)}\psi(\cdot,t))-R_{g(t)}|\nabla_{g(t)} \psi(\cdot,t)|_{g(t)}^2 -C(n)R_{g(t)}^2\psi(\cdot,t)^2\right)d\mu_{g(t)}\notag \\ 
&\ge \int_M \left(4|\Ric_{g(t)}|_{g(t)}-R_{g(t)}\right)|\nabla_{g(t)}\psi(\cdot,t)|^2d\mu_{g(t)}-C(n)\int_M R_{g(t)}^2\psi(\cdot,t)^2d\mu_{g(t)}.
\end{align}

By Cauchy inequality and a direct calculus using a special coordinate, one has
\begin{align}\label{aab2}
2|\Ric_{g(t)}|_{g(t)}^2\ge \frac{2}{n}R_{g(t)}^2.
\end{align}

By \eqref{aab1} and \eqref{aab2}, we have
\begin{align}
\partial_tE(t)\ge - \int_M |\Ric_{g(t)}||\nabla_{g(t)}\psi(\cdot,t)|^2d\mu_{g(t)}-C(n)\int_M R_{g(t)} ^2\psi(\cdot,t)^2d\mu_{g(t)}.
\end{align}
By Lemma \ref{thm66.2} (2) and Lemma \ref{prop_varphit} (1) proved above, we have
\begin{align*}
\partial_t E(t)\ge -\frac{C(n,h,p,\|\hat g\|_{W^{1,p}(M)})}{t^{\frac{n}{4p}+\frac{3}{4}}}\int_M |\nabla_{g(t)} \psi(\cdot,t)|^2d\mu_{g(t)}-C(n,h,p,\|\hat g\|_{W^{1,p}(M)},\tilde\varphi,a)\int_M R_{g(t)} ^2d\mu_{g(t)}.
\end{align*}

 In order to avoid the vanishing case, we consider $E(t)+1$, and we have
\begin{align*}
\partial_t\left(E(t)+1\right)\ge -\frac{C(n,h,p,\|\hat g\|_{W^{1,p}(M)})}{t^{\frac{n}{4p}+\frac{3}{4}}}\left(E(t)+1\right)-C(n,h,p,\|\hat g\|_{W^{1,p}(M)},\tilde\varphi,a)\int_M R_{g(t)} ^2d\mu_{g(t)}.
\end{align*}

Dividing both sides by $\left(E(t)+1\right)$, we have
\begin{align*}
\partial_t \log\left(E(t)+1\right)\ge -\frac{C(n,h,p,\|\hat g\|_{W^{1,p}(M)})}{t^{\frac{n}{4p}+\frac{3}{4}}}-C(n,h,p,\|\hat g\|_{W^{1,p}(M)},\tilde\varphi,a)\int_M R_{g(t)} ^2d\mu_{g(t)}.
\end{align*}

 By Lemma \ref{thm66.2} (3), $\int_M R_{g(t)} ^2d\mu_{g(t)}$ is integrable on $(0,T)$ and the integral is controlled by $C(n,h,p,\|\hat g\|_{W^{1,p}(M)})$. Thus taking integral we have
\begin{align}\label{E1}
\log\left(E(T)+1\right)-\log\left(E(t)+1\right)\ge C(n,h,p,\|\hat g\|_{W^{1,p}(M)},\tilde\varphi,a), \forall t\in[0,T].
\end{align}

 Recall that $\psi_T=\tilde\varphi$, since Ricci flow and $h$-flow are equivalent after a family of diffeomorphisms, we have 
\begin{align}
E(T)=\int_M | \nabla_{g(T)} \tilde\varphi(\cdot)|^2_{g(T)}d\mu_{g(T)}=\int_M | \nabla_{\bar g(T)} \tilde\varphi(\cdot)|^2_{\bar g(T)}d\mu_{\bar g(T)},
\end{align}
where $\bar g(T)$ is the metric at time $T$ of the $h$-flow $\bar g(t)$ given in Lemma \ref{thm6.2}. By Lemma \ref{thm6.2}, $h$ is $1+\epsilon(n)$-fair to $\bar g(t)$, thus we have
\begin{align}\label{E2}
E(T)\le C(n)\int_M | \nabla_h \tilde\varphi(\cdot)|^2_hd\mu_h\le C(n,h,\tilde\varphi)
\end{align}

Combining \eqref{E1} and \eqref{E2}, we have
\begin{align}\label{E3}
E(t)\le C(n,h,p,\|\hat g\|_{W^{1,p}(M)},\tilde\varphi,a), \forall t\in[0,T].
\end{align}

By Lemma \ref{prop_varphit}, \eqref{E3} and H\"older inequality, we have
\begin{align}\label{E4}
\|\varphi(\cdot,t)\|_{W^{1,\frac{p}{p-1}}(M)}\le \|\varphi(\cdot,t)\|_{C^0(M)}+C(n,h) E^{1/2}(t)\le C(n,h,p,\|\hat g\|_{W^{1,p}(M)},\tilde\varphi,a), \forall t\in[0,T].
\end{align}
which proves (4), thus the lemma is proved.
\end{proof}

\section{Ricci flow and scalar curvature lower bounds}
In this section, we study the scalar curvature lower bounds along Ricci flow. The main result in this section is:
\begin{lemma}\label{mthm2}
Let $M^n$ be a compact smooth manifold with a metric $\hat g\in W^{1,p}(M)$ $(n< p\le \infty)$. Suppose $R_{\hat g}\ge a$ in distirbutional sense for some constant $a$, and let $g(t),t\in(0,T_0]$ be the Ricci flow with initial metric $\hat g$. Then for any $t\in(0,T_0]$, there holds $R_{g(t)}\ge a\left(1-\frac{2a}{n}t\right)^{-1}$ pointwisely on $M$. 
\end{lemma}
\begin{proof}
We will mollify the initial metric $\hat g$ firstly, and consider a family of flows with the mollifying metrics as their initial metrics.

By Lemma \ref{rmkh}, we can take a smooth metric $h_1$ on $M$, such that $h_1$ is $2$-fair to $\hat g$.

Since $p>n$, by Sobolev embedding we have $\hat g\in C^0(M)$. Let $\hat g_\delta$ be the family of smooth metrics constructed in Lemma \ref{lm2.1}, such that $\hat g_\delta$ converges to $\hat g$ in $W^{1,p}$-norm taken with respect to $h_1$. To be precisely, we have
\begin{align}\label{w1ph1}
\lim_{\delta\to 0^+}\|\hat g_\delta-\hat g\|_{W^{1,p}(M);h_1}=0,
\end{align} 
where $\|\cdot\|_{W^{1,p}(M);h_1}$
is the $W^{1,p}$ norm for tensors on $M$ taken with repsect to $h_1$.

Then by Sobolev embedding, there exists $\delta_0\ge0$, such that,
\begin{align}
\|\hat g_\delta-\hat g\|_{C^0(M);h_1}\le \min\{\frac{\epsilon(n)}{8},\frac{1}{4}\},
\forall \delta\in(0,\delta_0],
\end{align}
where $\|\cdot\|_{C^0(M);h_1}$
is the $C^0$ norm for tensors on $M$ taken with repsect to $h_1$, and $\epsilon(n)$ is the dimensional constant given in Lemma \ref{thm66.2}.

Since $h_1$ is $2$-fair to $\hat g$, we have
\begin{align}
\|\hat g_\delta-\hat g\|_{C^0(M);\hat g}\le \min\{\frac{\epsilon(n)}{4},\frac{1}{2}\},
\forall \delta\in(0,\delta_0],
\end{align}

Since $\|\hat g_{\delta_0}-\hat g\|_{C^0(M);\hat g}\le  \frac{1}{2}$, we have
\begin{align}
\|\hat g_\delta-\hat g\|_{C^0(M);\hat g_{\delta_0}}\le \frac{\epsilon(n)}{2},
\forall \delta\in(0,\delta_0].
\end{align}

Thus we can let $h=\hat g_{\delta_0}$, and then $h$ is $(1+\frac{\epsilon(n)}{2})$-fair to all of $\hat g_\delta(0<\delta\le \delta_0)$. For convenience, $\|\cdot\|_{C^0(M)}$ and $\|\cdot\|_{W^{k,q}(M)}$ will denote the $C^0$ norm and the $W^{k,q}$ norm respectively for tensors on $M$ taken with repsect to $h$.

Since $h$ is $(1+\frac{\epsilon(n)}{2})$-fair to all of $\hat g_\delta(0<\delta\le \delta_0)$, for each smooth metric $\hat g_\delta$ we consider the Ricci flow $g_\delta(t)$  given in Lemma \ref{mainsec6} with initial metric $\hat g_\delta$. I is known that by letting $\delta$ converge to $0^+$, $g_\delta(t)(0\in (0,T_0])$ converge to a Ricci flow $g(t)(0\in (0,T_0])$ such that $\lim_{t\to 0}d_{GH}((M,g(t)),(M,\hat g))=0$, where $d_{GH}$ is the Gromov-Hausdorff distance.

For any $T\in (0,T_0]$ and any nonnegative $\tilde \varphi\in C^\infty(M)$, we will prove
\begin{align}
\int_M \left(R_{g(T)}-a\left(1-\frac{2a}{n}t\right)^{-1}\right)\tilde \varphi d\mu_{g(T)}\ge0,
\end{align}
which is sufficent to give $R_{g(t)}\ge a\left(1-\frac{2a}{n}t\right)^{-1}$ pointwisely on $M$.

To do this, for each $\delta\in (0,\delta_0]$ we consider the auxilliary functions $\varphi_\delta$ given in Lemma \ref{prop_varphit}, such that 
\begin{enumerate}
	\item $\varphi_\delta(\cdot,T)=\tilde \varphi(\cdot),\  {\text{on}}\ M$.
	\item For any constant $a$, $\int_M \left(R_{g_\delta(t)} -a\left(1-\frac{2a}{n}t\right)^{-1}\right) \varphi_\delta(\cdot,t)d\mu_{g_\delta(t)}$ is monotonously increasing with respect to $t$ (if $a>0$, then we require $t\le T\le T_0<\frac{n}{2a}$).
	\item $\varphi_\delta(\cdot,t)\le C(n,h,p,\|\hat g_\delta\|_{W^{1,p}(M)},\|\tilde\varphi\|_{L^\infty},a)$, $\forall t\in[0,T].$
	\item  $\|\varphi_\delta(\cdot,t)\|_{W^{1,\frac{n}{n-1}}(M)}\le C(n,h,p,\|\hat g_\delta\|_{W^{1,p}(M)},\|\tilde\varphi\|_{L^\infty},a)$,  $\forall t\in[0,T].$
\end{enumerate}

By \eqref{w1ph1}, we have
\begin{align}\label{unifm}
\|\hat g_\delta\|_{W^{1,p}(M)}\le 2\|\hat g\|_{W^{1,p}(M)}.
\end{align}

Recall $h=\hat g_{\delta_0}$ only depends  on $\hat g$. Thus by \eqref{unifm}, the  estimate above could be uniformized as
 \begin{align}\label{liftphi}
 \varphi_\delta(\cdot,t)\le C(n,p,\hat g,\tilde\varphi,a), \forall t\in[0,T],
 \end{align}
 and
 \begin{align}\label{liftphi1}
 \|\varphi_\delta(\cdot,t)\|_{W^{1,\frac{n}{n-1}}(M)}\le C(n,p,\hat g,\tilde\varphi,a), \forall t\in[0,T].
 \end{align}

 Let us estimate the integral
 \begin{align}
 \int_M \left(R_{g_\delta(t)} -a\left(1-\frac{2a}{n}t\right)^{-1}\right) \varphi_\delta(\cdot,t)d\mu_{g_\delta(t)},
 \end{align}
 firstly.

 By the monotonicity we have
 \begin{align}\label{mnt1''}
 \int_M \left(R_{g_\delta(t)} -a\left(1-\frac{2a}{n}t\right)^{-1}\right) \varphi_\delta(\cdot,t)d\mu_{g_\delta(t)}\ge  \int_M \left(R_{\hat g_\delta} -a\right) \varphi_\delta(\cdot,0)d\mu_{\hat g_\delta}.
 \end{align}

To estimate $\int_M \left(R_{\hat g_\delta} -a\right) \varphi_\delta(\cdot,0)d\mu_{\hat g_\delta}$, by Lemma \ref{lm2.2}, we have
\begin{align}\label{d1}
\left|\int_M R_{\hat g_\delta}\varphi_\delta(\cdot,0) d\mu_{\hat g_\delta}-\langle R_{\hat g},\varphi_\delta(\cdot,0) \rangle\right|\le \Psi(\delta|\hat g) \|\varphi_\delta(\cdot,0) \|_{W^{1,\frac{n}{n-1}}(M)}, \forall \varphi\in C^\infty(M),
\end{align}
where $\Psi(\delta|\hat g)$ is a positive function such that $\lim_{\delta\to0^+}\Psi(\delta|\hat g)=0$ for any fixed $\hat g$, and $\Psi(\delta|\hat g)$ varies from line to line.

Moreover, by Sobolev embedding we have $\lim_{\delta\to0^+}\left\|\frac{d\mu_{\hat g_\delta}}{d\mu_{\hat g}}-1\right\|_{C^0(M)}=0$, thus by H\"older inequality, we have
\begin{align}\label{d2}
\left|\int_M\varphi_\delta(\cdot,0) d\mu_{\hat g_\delta}-\int_M\varphi_\delta(\cdot,0) d\mu_{\hat g}\right|
&= \left|\int_M\varphi_\delta(\cdot,0) \left(\frac{d\mu_{\hat g_\delta}}{d\mu_{\hat g}}-1 \right)d\mu_{\hat g}\right|\notag\\
&\le \left\|\frac{d\mu_{\hat g_\delta}}{d\mu_{\hat g}}-1\right\|_{C^0(M)}\int_M|\varphi_\delta(\cdot,0)| d\mu_{\hat g}
\notag\\
&\le  C(n,\hat g)\left\|\frac{d\mu_{\hat g_\delta}}{d\mu_{\hat g}}-1\right\|_{C^0(M)}\|\varphi_\delta(\cdot,0)\|_{W^{1,\frac{n}{n-1}}(M)}\notag\\
&\le  \Psi(\delta|\hat g)\|\varphi_\delta(\cdot,0)\|_{W^{1,\frac{n}{n-1}}(M)}.
\end{align}

By triangular inequality, 
\begin{align}\label{d3'}
&\left|\int_M (R_{\hat g_\delta}-a)\varphi_\delta(\cdot,0) d\mu_{\hat g_\delta}-\langle R_{\hat g}-a,\varphi_\delta(\cdot,0)\rangle\right|\notag\\
\le& \left|\int_M R_{\hat g_\delta}\varphi_\delta(\cdot,0) d\mu_{\hat g_\delta}-\langle R_{\hat g},\varphi_\delta(\cdot,0)\rangle\right|+|a|\left|\int_M\varphi_\delta(\cdot,0) d\mu_{\hat g_\delta}-\int_M\varphi_\delta(\cdot,0) d\mu_{\hat g}\right|.
\end{align}

Combining  \eqref{d1}-\eqref{d3'}, we have
\begin{align}\label{d3}
&\left|\int_M (R_{\hat g_\delta}-a)\varphi_\delta(\cdot,0) d\mu_{\hat g_\delta}-\langle R_{\hat g}-a,\varphi_\delta(\cdot,0)\rangle\right|\notag\\
\le& \Psi(\delta|\hat g) \|\varphi_\delta(\cdot,0)\|_{W^{1,\frac{n}{n-1}}(M)}, \forall \tilde\varphi\in C^\infty(M),\tilde\varphi\ge0\forall\delta\in(0,\delta_0].
\end{align}

Since we have assumed $R_{\hat g}\ge a$ in distributional sense, we have 
\begin{equation}\label{d}
\langle R_{\hat g}-a,\varphi_\delta(\cdot,0)\rangle \ge 0, \forall \varphi_\delta(\cdot,0)\in C^\infty(M),\varphi_\delta(\cdot,0)\ge0.
\end{equation}

Combining \eqref{mnt1''}, \eqref{d3} and \eqref{d}, we have
\begin{align}\label{aaa}
\int_M \left(R_{g_\delta(t)} -a\left(1-\frac{2a}{n}t\right)^{-1}\right) \varphi_\delta(\cdot,t)d\mu_{g_\delta(t)}\ge \int_M (R_{\hat g_\delta}-a)\varphi_\delta(\cdot,0) d\mu_{\hat g_\delta}\ge -\Psi(\delta|\hat g) \|\varphi_\delta(\cdot,0)\|_{W^{1,\frac{n}{n-1}}(M)}.
\end{align}

By \eqref{liftphi1} and \eqref{aaa}, we have
\begin{align}\label{aaa1}
\int_M \left(R_{g_\delta(t)} -a\left(1-\frac{2a}{n}t\right)^{-1}\right) \varphi_\delta(\cdot,t)d\mu_{g_\delta(t)}
&\ge -C(n,p,\hat g,\tilde\varphi,a)\Psi(\delta|\hat g)\notag\\
&\ge  -\Psi(\delta|n,p,\hat g,\tilde\varphi,a), \forall t\in [0,T],\delta\in (0,\delta_0].
\end{align}

In particular, letting $t=T$ in \eqref{aaa1}, we have
\begin{align}\label{aaa111}
\int_M \left(R_{g_\delta(T)} -a\left(1-\frac{2a}{n}T\right)^{-1}\right) \tilde \varphi d\mu_{g_\delta(T)}\ge  -\Psi(\delta|n,p,\hat g,\tilde\varphi,a), \forall \delta\in (0,\delta_0].
\end{align}

where $\Psi(\delta|n,p,\hat g,\tilde\varphi,a)$ denotes a positive function such that $\lim_{\delta\to0^+}\Psi(\delta|n,p,\hat g,\tilde\varphi,a)=0$ for any fixed $n,p,\hat g,\tilde\varphi$ and $a.$

By Simon's estimate, Lemma \ref{thm2.2}, as $\delta$ tends to $0$, $g_\delta(T)$ smoothly converges to $g(T)$. Thus taking limit in \eqref{aaa111}, we have
\begin{align}\label{finalR}
\int_M (R_{g(T)}-a\left(1-\frac{2a}{n}T\right)^{-1})\tilde\varphi d\mu_{g(T)} \ge 0,\forall T\in(0,T_0],\forall \tilde\varphi \in C^\infty(M),\tilde\varphi\ge0.
\end{align}

Recall that $g(t)$ is a smooth metric for $t\in(0, T_0]$ and $R_{g(t)}$ is well defined pointwisely on $M$, thus by \eqref{finalR} we have $R_{g(t)}\ge a\left(1-\frac{2a}{n}t\right)^{-1}$ pointwisely on $M$ for any $t\in (0,T_0]$, which completes the proof of the theorem.
\end{proof}

\section{proof of Theorem \ref{thmY2}}
In this section, we prove Theorem \ref{thmY2}.  Let us restate it as follows: 

\begin{theorem}\label{thmY2'}
Let $M^n$ be a compact manifold with $\sigma(M)\le0$ and $\hat g$ be a  $W^{1,p}(n< p\le \infty)$ metric on $M$ with unit volume  such that $R_{\hat g}\ge \sigma(M)$ in distributional sense. Then $(M,\hat g)$ is isometric to an Einstein manifold with scalar curvature equaling to $\sigma(M)$.
\end{theorem}

For any smooth mainifold $M$, its Yamabe invariant $\sigma(M)$ is defined as:
\begin{align}\label{dnfsigma}
\sigma(M):=\sup_{\mathcal{C}}\inf_{g\in \mathcal{C}}\frac{\int_M R_gd\mu_g}{(\Vol(M,g))^{(n-2)/2}},
\end{align}
where $\mathcal{C}$ is the set consists of every conformal class of Riemannian metrics on $M$, $R_{g}$ is the scalar curvature of $g$, and $\Vol(M,g)$ is the volume of $(M,g)$.

Roughly speaking, the basic idea of proving Themrem \ref{thmY2'} is to flow the initial metric $\hat g$. The flow at positive time are smooth, thus we can prove  they are Einstein by using Theorem \ref{thmY222}.

\begin{proof}

\textbf{Step A: we consider a normalized Ricci flow $\mathring g(t)$ which keeps unit volume and prove that $R_{\mathring g(t)}\ge \sigma(M)$.}

In order to apply Theorem \ref{thmY222} to a flow, the volume along the flow should be ketp unit. Therefore we consider such a normalized Ricci flow (see \cite{LT21}):
\begin{align}\label{ring g}
\mathring g(t)=\left(\Vol(M,g(t))\right)^{-2/n}g(t), t\in (0,T_0],
\end{align}
where $g(t)(t\in (0,T_0])$ is the Ricci flow with initial metric $\hat g$ (see Lemma \ref{mainsec6}).

Then by a standard calculation, we have
\begin{align}\label{unitvol}
\Vol(M,\mathring g(t))\equiv 1, t\in (0,T_0],
\end{align}
and 
\begin{align}\label{RR1}
R_{\mathring g(t)}=\left(\Vol(M,g(t))\right)^{2/n}R_{g(t)}, t\in (0,T_0].
\end{align}

We want to prove $R_{\mathring g(t)}\ge \sigma(M)$. To do this, recall that we have assumed $R_{\hat g}\ge \sigma(M)$ in distributional sense, by Lemma \ref{mthm2} we have
\begin{align}\label{Rgtg}
R_{g(t)}\ge \sigma(M)\left(1-\frac{2\sigma(M)}{n}t\right)^{-1},
\end{align} 
pointwisely on $M$.

Thus we need to compare $\left(\Vol(M,g(t))\right)^{2/n}$  with $\left(1-\frac{2\sigma(M)}{n}t\right)$.

By \eqref{mnt4}, we calculate that
\begin{align}\label{Rgtg1}
\frac{d}{dt}\Vol(M,g(t))&=\frac{d}{dt}\int_Md\mu_{g(t)}\notag\\
&=\int_M(-R_{g(t)})d\mu_{g(t)}.
\end{align}

Combining \eqref{Rgtg} and \eqref{Rgtg1}, we have
\begin{align}\label{Rgtg2}
\frac{d}{dt}\Vol(M,g(t))
&\le -\int_M \sigma(M)\left(1-\frac{2\sigma(M)}{n}t\right)^{-1}d\mu_{g(t)}\notag\\
&=-\sigma(M)\left(1-\frac{2\sigma(M)}{n}t\right)^{-1}\Vol(M,g(t)).
\end{align}

Then we have
\begin{align}
\frac{d}{dt}\log\Vol(M,g(t)\le -\sigma(M)\left(1-\frac{2\sigma(M)}{n}t\right)^{-1}.
\end{align}

Taking integral on $(0,t)$ we have
\begin{align}
\log\Vol(M,g(t)-\log\Vol(M,\hat g)&\le -\sigma(M)\frac{n}{2\sigma(M)}\log\left(1-\frac{2\sigma(M)}{n}t\right)\notag\\
&=\frac{n}{2}\log\left(1-\frac{2\sigma(M)}{n}t\right).
\end{align}

Recall that we have assumed $\hat g$ with unit volume, thus $\log\Vol(M,\hat g)=0$, and we have
\begin{align}\label{RR2}
\Vol(M,g(t))\le \left(1-\frac{2\sigma(M)}{n}t\right)^{n/2}.
\end{align}

Combining \eqref{RR1}, \eqref{Rgtg} and \eqref{RR2}, we have
\begin{align}\label{ring R}
R_{\mathring g(t)}\ge \sigma(M),
\end{align}
which completes the step A.

\textbf{Step B: we prove the normalized Ricci flow $\mathring g(t)$ is independent of $t$.}

Since $\mathring g(t)$ is a smooth Riemannian metric on $M$, and $\sigma(M)$ is nonpositive, by the classical theorem, Theorem \ref{thmY222}, we have that $\mathring g(t)$ is Einstein with $\Ric_{\mathring g(t)}=\frac{\sigma(M)}{n}\mathring g(t)$.

Thus by \eqref{ring g}, we have
\begin{align}\label{ring g1}
\Ric_{g(t)}=\Ric_{\mathring g(t)}=\frac{\sigma(M)}{n}\mathring g(t).
\end{align}

On the other hand, by \eqref{ring g} we have
\begin{align}\label{ring g2}
\frac{\partial}{\partial t}\mathring g(t)&=-\frac{2}{n}\left(\Vol(M,g(t))\right)^{-2/n-1}\frac{d}{dt}\Vol(M,g(t))g(t)+\frac{\partial}{\partial t}  g(t)\notag\\
&=-\frac{2}{n}\left(\Vol(M,g(t))\right)^{-2/n-1}\frac{d}{dt}\Vol(M,g(t))g(t)-2\Ric_{g(t)}.
\end{align}

By \eqref{ring g1} and \eqref{ring g2}, we have
\begin{align}\label{ring g33}
\frac{\partial}{\partial t}\mathring g(t)
&=-\frac{2}{n}\left(\Vol(M,g(t))\right)^{-2/n-1}\frac{d}{dt}\Vol(M,g(t))g(t)-\frac{2\sigma(M)}{n}\mathring g(t)\notag\\
&=f(t)g(t),
\end{align}
where $f(t)=-\frac{2}{n}\left(\left(\Vol(M,g(t))\right)^{-2/n-1}\frac{d}{dt}\Vol(M,g(t))+\sigma(M)\right)$ is a constant on $M$ for any fixed $t\in(0,T_0]$.

Thus $\mathring g(t)$ must be self-similar, that is, we have
\begin{align}
\mathring g(t_1)=F(t_1,t_2)\mathring g(t_2),\ \forall t_1,t_2\in(0,T_0],
\end{align}
where $F$ is a constant on $M$  depends only on $t_1,t_2$.

Thus their Ricci curvature satisfies
\begin{align}\label{ring g3}
\Ric_{\mathring g(t_1)}=\Ric_{\mathring g(t_2)}=\frac{\sigma(M)}{n}\mathring g(t).
\end{align}

By \eqref{ring g1} and \eqref{ring g3}, we have 
\begin{align}\label{ring g4}
\frac{\sigma(M)}{n}\mathring g(t_1)=\Ric_{\mathring g(t_1)}=\Ric_{\mathring g(t_2)}=\frac{\sigma(M)}{n}\mathring g(t_2),
\end{align}
which gives
\begin{align}\label{ring g5}
\mathring g(t_1)=\mathring g(t_2),\ \forall t_1,t_2\in(0,T_0],
\end{align}
which completes the step B.

\textbf{Step C: we prove the initial metric is isometric to an Einstein manifold with scalar curvature $\sigma(M)$.}

Thus by \eqref{ring g33} and \eqref{ring g5}, the function $f(t)$ in \eqref{ring g33} is identically zero on $M\times (0,T_0]$.

Recall $f(t)=-\frac{2}{n}\left(\left(\Vol(M,g(t))\right)^{-2/n-1}\frac{d}{dt}\Vol(M,g(t))+\sigma(M)\right)$, thus $f\equiv0$ gives an ODE:
 \begin{align}\label{ODEgt}
\left(\Vol(M,g(t))\right)^{-2/n-1}\frac{d}{dt}\Vol(M,g(t))+\sigma(M)=0,\ t\in(0,T_0].
\end{align}

Since we have assumed that $\hat g$ has unit volume, by Lemma \ref{mainsec6} (1), we have  
\begin{align}\label{ODEgt1} 
\lim_{t\to 0^+}\Vol(M,g(t))=\Vol(M,\hat g)=1.
\end{align}

Solving the ODE \eqref{ODEgt} and \eqref{ODEgt1}, we have
\begin{align}\label{RR3}
\Vol(M,g(t))=\left(1-\frac{2\sigma(M)}{n}t\right)^{n/2}.
\end{align}

By \eqref{ring g}, we have
\begin{align}\label{RR4}
\mathring  g(t)=\left(1-\frac{2\sigma(M)}{n}t\right)^{-1}g(t).
\end{align}

Since Lemma \ref{mainsec6} (1) tells that
\begin{align}\label{RR5}
\lim_{t\to 0}d_{GH}((M,g(t)),(M,\hat g))=0.
\end{align}

By \eqref{RR4} and \eqref{RR5} we have $\mathring g(t)$ converges to
\begin{align}\label{RR6}
\lim_{t\to 0}d_{GH}((M,\mathring g(t)),(M,\hat g))=0.
\end{align}

However, by \eqref{ring g5}, $\mathring g(t)$ does not depend on $t$, thus \eqref{RR6} gives
\begin{align}\label{RR66}
d_{GH}((M,\mathring g(t)),(M,\hat g))=0,\ \forall t\in(0,T_0].
\end{align}

Thus by \eqref{ring g3} and \eqref{RR66}, we have that $\hat g$ is isometric to an Einstein manifold with scalar curvature $\sigma(M)$, which completes the proof of the theorem.
\end{proof}
\section*{Acknowledgements}
The work is supported by  National Key R\&D Program of China: 2022YFA1005500.

\bibliographystyle{plain}

\end{document}